\documentclass{amsart}

\usepackage[all]{xy}   

\usepackage{hyperref} 

\usepackage[usenames,dvipsnames]{xcolor}

\theoremstyle{plain}
\newtheorem{lem}{Lemma}[section]
\newtheorem{cor}[lem]{Corollary}
\newtheorem{prop}[lem]{Proposition}
\newtheorem{thm}[lem]{Theorem}

\theoremstyle{definition}
\newtheorem{ex}[lem]{Example}
\newtheorem{rem}[lem]{Remark}
\newtheorem{dfn}[lem]{Definition}

\newcommand{\Z}{\mathbb{Z}}               
\newcommand{\Q}{\mathbb{Q}}              

\newcommand{\al}{\alpha}    
\newcommand{\be}{\beta}
\newcommand{\ga}{\gamma}
\newcommand{\de}{\delta}  
\newcommand{\ep}{\epsilon}
\newcommand{\la}{\lambda}   
\newcommand{\om}{\omega}
\newcommand{\La}{\Lambda}  
\newcommand{\Th}{\Theta}      
\newcommand{\De}{\Delta}

\newcommand{\impl}{\;\Longrightarrow\;}   
\newcommand{\eqst}{\;\Longleftrightarrow\;}   
\newcommand{\ra}{\longrightarrow}             
\newcommand{\id}{\mathrm{id}}      
       
\newcommand{\trg}{\triangleleft}       
\newcommand{\Ln}{\La_{\emptyset}}      
\newcommand{\edg}{{\;{-}\!\!{-}\;}}             
\newcommand{\Gg}{\mathcal{G}}           
\newcommand{\SA}{\mathcal{Z}}     
\newcommand{\sgn}{\mathrm{sgn}}     
\newcommand{\ch}{\mathrm{ch}}      
\newcommand{\cch}{\mathfrak{ch}}    
\newcommand{\Lie}{\mathfrak{g}}   
\newcommand{\Crt}{\mathfrak{h}}    
\newcommand{\Hom}{\mathrm{Hom}}

\begin{document}

\title[A Riemann-Roch type theorem for moment graphs]{A Riemann-Roch type theorem for twisted fibrations of moment graphs}

\author[M.~Lanini]{Martina Lanini}
\address[Martina Lanini]{Dipartimento di Matematica, University of Rome Tor Vergata, Via della Ricerca Scientifica 1, 00133, Rome, Italy}
\email{lanini@mat.uniroma2.it}
\urladdr{https://sites.google.com/site/martinalanini5/home}

\author[K.~Zainoulline]{Kirill Zainoulline}
\address[Kirill Zainoulline]{Department of Mathematics and Statistics, University of Ottawa, 150 Louis-Pasteur, Ottawa, ON, K1N 6N5, Canada}
\email{kirill@uottawa.ca}
\urladdr{http://mysite.science.uottawa.ca/kzaynull/}

\subjclass[2010]{14F05, 14F43, 14M15}
\keywords{equivariant cohomology, Chern character, Riemann-Roch theorem, moment graph, structure algebra}

\begin{abstract}  In the present paper we extend the Riemann-Roch formalism to structure algebras of moment graphs.
We introduce and study the Chern character and pushforwards for twisted fibrations of moment graphs. 
We prove an analogue of the Riemann-Roch theorem for moment graphs. As an application, we obtain the Riemann-Roch type theorem for equivariant $K$-theory of some
Kac-Moody flag varieties.
\end{abstract}

\maketitle

\tableofcontents


\section{Introduction}
Moment graphs are combinatorial gadgets arising as labelled one-skeleta of torus actions on (non-necessarily smooth) varieties. In most of the cases they encode all the necessary data to describe equivariant generalized cohomology theories in different settings (see, for instance, \cite{CS74}, \cite{GKM}, \cite{HHH}, \cite{DLZ}). For example, the explicit combinatorial description which one obtains using localization techniques has allowed Tymoczko \cite{Ty08} to study the symmetric group \emph{dot-action} on Hessenberg varieties.

If the variety also admits a stratification which is compatible with the torus action, then the closure inclusion relation among strata induces an orientation on the edges of the corresponding moment graph. Braden and MacPherson \cite{BM01} showed that moment graphs arising in this way can be applied to determine the stalks of the equivariant intersection cohomology complexes. In this occasion they defined the notion of sheaves on moment graphs. Braden-MacPherson's result was further extended and developed by Fiebig and Williamson \cite{FW} to determine the stalks of indecomposable parity sheaves. Thanks to representation theoretical role of intersection cohomology and parity sheaves, the moment graph theory has become an important tool in modern representation theory (see the survey~\cite{Fi13} for some of the applications). Moreover,  it provides a realization of the category of Soergel bimodules associted with any Coxeter group (see \cite{Fi08}), as well as allowing to  categorify some properties of Kazhdan-Lusztig polynomials (see \cite{L12, L15}). 

While developing tools needed for the representation theoretical applications, Fiebig realized that it was possible to carry the theory of sheaves on moment graphs in an axiomatic way, so that several sheaf theoretical notions where adapted to the moment graph setting without any need of an actual geometry. This allows one, for example, to define the analogue of IC-sheaves for a moment graph associated to any Coxeter system (the so-called BMP-sheaves), leading Fiebig to the above-mentioned moment graph realization  of Soergel bimodules. On the other hand, pullbacks of moment graph morphisms between the corresponding categories of sheaves (see \cite{L12}) provide a fundamental tool to categorify equalities among Kazhdan-Lusztig polynomials.

In this paper, we aim to further develop moment graph analogues of classical topological/geometric tools. 
Namely, we intend to extend the general Riemann-Roch formalism to the moment graph settings. 
Observe that given a moment graph it is indeed possible to construct the corresponding additive and multiplicative structure algebras. In the case of a moment graph arising as a one-skeleton of a torus action on a nice enough variety, these coincide with the equivariant cohomology and K-theory, respectively. 
Given a moment graph morphism (see Definition \ref{Defn:MGmorph}), one would like to define pullbacks and pushfowards between the corresponding structure algebras. We therefore introduce the notion of (twisted) pullback and of pushfoward along a (twisted) fibration. For the reader familiar with flag varieties, the latter notion is the moment graph generalization of the fibration corresponding to the quotient morphism from the variety of complete flags to a variety of partial flags.
As a first consequence of our constructions, we are able to define a divided difference operator on the structure algebra of a moment graph having a special matching, and satisfying some extra assumption (see Definition \ref{ex:specialmatching}). Special matchings have played an important role  in proving special cases of  the invariance conjecture for Kazhdan-Lusztig polynomials (see, for example, \cite{B}), so that we believe that the notion of divided difference operator in this setting might be of some interest to algebraic combinatorialists.

In order to state and prove the Riemann-Roch type theorem in the moment graph setting, we introduce the truncated Chern character $\cch_i\colon \SA_m(\Gg) \to \SA_a(\Gg)$ between the structure algebras of the moment graph $\Gg$ by gluing the respective exponential maps over the fixed loci (see \S\ref{sec:chern}). Observe that this idea has been successfully applied before in different contexts, e.g.
in the settings of the equivariant Riemann-Roch theorem for equivariant $K$-theory (see \cite{EG00}), for bivariant operational $K$-theory (see \cite{AGP}), for cohomological operations on equivariant oriented cohomology of flag varieties (see \cite{Za20}).
We then show that $\cch$ commutes with characteristic maps and pullbacks. 
As for pushforwards, our main result (Theorem~\ref{thm:mainRR}) -- the analogue of the Riemann-Roch theorem for moment graphs -- 
says that the push-forward $\pi^\xi_\ast$ for a twisted fibration $\pi^\xi$ commutes with the Chern character up to a multiple by the respective Todd genus $td_i^\xi$ of $\Gg$ 
(another interesting geometric invariant of the moment graph). Namely, we prove that 
for any $z\in \SA_m(\Gg)$ there is the following Riemann-Roch type formula:
\[
\pi^\xi_\ast \big(\cch_i(z)\cdot td_i^\xi(\Gg)\big)=\cch_i(\pi^\xi_\ast(z)).
\] 
Applying our result to the case of Kac-Moody flag varieties, we obtain the respective equivariant Riemann-Roch theorem for fibrations of the type 
$G/B\rightarrow G/P$, where $G\supseteq P\supseteq B$ are a Kac-Moody group, its parabolic and Borel subgroups, respectively.

\vspace{3mm}

\paragraph{\it Organization of the paper}
In Section~\ref{sec:momgr}  we recall definitions and provide examples of moment graphs and of morphisms among them, we then discuss quotients of moment graphs. In Section~\ref{sec:stralg} we study structure algebras and their behaviour with respect to characteristic maps and filtrations; we introduce twisted pull-backs. Section~\ref{sec:pushpull} is dedicated to the construction of push-forwards maps induced by twisted fibrations of moment graphs (Proposition~\ref{prop:push}); we prove the projection formula (Corollary~\ref{cor:projform}) and produce analogues of push-pull operators on moment graphs. In Section~\ref{sec:chern} we introduce the Chern character between structure algebras of moment graphs (Proposition~\ref{thm:main}); we study its properties with respect to characteristic maps, pull-backs and forgetful maps (Lemma~\ref{lem:charm}, \ref{lem:forgt} and~\ref{lem:cpull}). We state and prove our main result (Theorem~\ref{thm:mainRR}). In Section~\ref{sec:KacM} we apply our theorem to Kac-Moody flag varieties. 

\paragraph{\it Acknowledgements} M.L. acknowledges the MIUR Excellence Department Project awarded to the Department of Mathematics, University of Rome Tor Vergata, CUP E83C18000100006, and the PRIN2017 CUPE84|19000480006. K.Z. was partially supported by the NSERC Discovery grant RGPIN-2015-04469, Canada.


\section{Moment graphs and their quotients}\label{sec:momgr}

In the present section we recall the definition and provide examples of moment graphs and morphisms among them. We introduce the notion of a monodromy of a moment graph in~\ref{dfn:gmond}. We then discuss quotients of moment graphs modulo certain equivalence relations. In particular, we study relations and quotients associated to special matchings in ~\ref{ex:specialmatching}.

\subsection{Definitions and examples}
Let $\La$ be a lattice (a free abelian group of finite rank). We denote by $\Ln$ the subset of non-zero elements of $\La$. We recall the definition of a moment graph on a lattice  $\La$ from \cite{L12}:

\begin{dfn}
The data $\Gg=\big((V,\le), l\colon E\to \Ln\big)$ is called a {\it moment graph} on the lattice $\La$ if
\begin{itemize}

\item[(MG1)] 
$V$ is a set of vertices together with a partial order `$\le$', i.e.\! we are given a poset $(V,\le)$.

\item[(MG2)] 
$E$ is a set of directed edges labelled by a nonzero element of $\La$ via the label function $l$, i.e.\! $E\subset V\times V$ with an edge $(v,w)\in E$ denoted by $v\to w$ and labelled by $l(v \to w)\in \Ln$.

\item[(MG3)] 
For any edge $v\to w\in E$, we have $v\le w$, $v\neq w$, i.e.\! direction of edges respects the partial order.
\end{itemize}
\end{dfn}

\begin{rem}
Observe that (MG2) and (MG3) imply that the graph does not have multiple edges or self-loops: (MG2) disallows several edges between the same two vertices in the same direction, and (MG3) disallows pairs of edges between two vertices in the opposite directions and self-loops. Also, (MG3) implies that $\Gg$ has no directed cycles. Observe also that the direction of edges in $\Gg$ is uniquely determined by the partial order `$\le$'.
\end{rem}

\begin{ex}\label{ex:Bruhat} 
Let $W$ be a {\it real finite reflection group} in the sense of \cite[I.1]{Hu90}. Let $\Phi$ be the associated root system together with a subset of simple roots $\Pi$ and the decomposition $\Phi=\Phi_+\amalg \Phi_-$ into positive and negative roots. Consider the usual Bruhat poset $(W,\le)$ of \cite[II.5.9]{Hu90}. The data 
\[
V:=W,\quad E:=\{w\to s_\al w\mid  w\le s_\al w,\, \al\in \Phi_+\}\quad \text{ and }\quad l(w\to s_\al w):=\al,
\]
where $s_\al$ is the reflection corresponding to the positive root $\al$, define a moment graph on the root lattice $\La_r=Span_\Z(\Phi)$ called the {\it Bruhat moment graph} and denoted $\Gg(W)$. 

Moreover, let $\Th$ be a subset of $\Pi$ and let $W_\Th$ be the (parabolic) subgroup generated by reflections corresponding to the roots from $\Th$. Let $W^\Th$ denote the subset of minimal coset representatives of $W/W_\Th$ (such representatives are unique). Consider the restricted Bruhat poset $(W^\Th,\le)$. The data
\[
V:=W^\Th,\quad E:=\{w\to \overline{s_\al w}\mid  w\le \overline{s_\al w},\, \al\in \Phi_+\}\quad \text{ and }\quad l(w\to \overline{s_\al w}):=\al,
\]
where $\overline{w}$ denotes the minimal coset representative of the coset $wW_\Th$, define a moment graph on the same root lattice $\La_r$ called the {\it parabolic Bruhat moment graph} and denoted $\Gg(W^\Th)$.

As examples one can also take full subgraphs of $\Gg(W^\Th)$ corresponding to Bruhat intervals $[v,w]$ on $W^\Th$ with vertices $\{u\in W^\Th \mid v \le u \le w\}$, or graphs corresponding to double cosets of $W$ (see \cite[\S4]{DLZ} for details).
\end{ex}

\subsection{Moment graph morphisms} 
Given an edge $x\to y$ or $y\to x$ we will use the notation $x\edg y$ if we are only considering the underlying edge without orientation.
We recall the definition of moment graph morphisms and isomorphisms. 

\begin{dfn}\label{Defn:MGmorph}(cf.~\cite[Definition~2.3]{L12})
A morphism between two moment graphs on $\La$\[
f\colon \big((V,\le), l\colon E\to \Ln\big) \ra \big((V',\le'), l'\colon E'\to \Ln \big)
\]
is given by a collection of maps $(f_V, \{f_{l,v}\}_{v\in V})$, where
\begin{itemize}
\item[(MR1)] 
$f_V\colon V \ra V'$ is a morphism of posets such that
\[
v \edg w\in E \impl f_V(v) \edg f_V(w)\in E' \;\text{ or }\; f_V(v)=f_V(w).
\]
\end{itemize}

Given an edge $v \edg w\in E$ such that $f_{V}(v)\neq f_{V}(w)$, we set 
\[
f_E(v \edg w):= f_V(v) \edg f_V(w).
\]

\begin{itemize}
\item[(MR2)] 
$f_{l,v}$ is a $\Z$-linear automorphism of $\La$ for each $v\in V$ such that \\
if $v\edg w \in E$ and $f_{V}(v)\neq f_{V}(w)$, then 
\item[(a)] 
$f_{l,v}(l(v \edg w))=\pm l'(f_E(v \edg w))$,
\item[(b)]  
$\pi \circ f_{l,v} =\pi \circ f_{l,w}$,
where $\pi$ is the canonical quotient map 
\[
\pi\colon \La \ra \La/l'(f_{E}(v \edg w))\Z.
\]
\end{itemize}
\end{dfn}

\begin{lem}(cf.~\cite[Lemma~3.6]{L15}) 
A morphism $f=(f_V, \{f_{l,v}\}_{v\in V})$ between two moment graphs $\Gg = \big((V,\le), l\colon E\to \Ln\big)$ and $\Gg' = \big((V',\le'), l'\colon E\to \Ln)$ on $\La$ is an isomorphism if and only if the following two conditions hold:
\begin{itemize}
\item[(ISO1)]$f_V$ is bijective,
\item[(ISO2)]for each $v'\to w'\in E'$ there exists exactly one $v\to w\in E$ such that $f_V(v) = v'$ 
and $f_V(w) = w'$.
\end{itemize}
\end{lem}

We will need the following
\begin{dfn}\label{dfn:gmond}
A collection $\xi=\{\xi_v\}_{v\in V}$ of automorphisms of the lattice $\La$ is called a {\it $\Gg$-monodromy} if 
\begin{equation}\label{eq:mon}
\xi_v(\la)-\xi_w(\la)\in l(v\to w)\Z,\quad \forall\; v\to w\in E\text{ and }\la\in \La. 
\end{equation}
\end{dfn}
Observe that a $\Gg$-monodromy $\{\xi_v\}_{v\in V}$ which satisfies \rm{(MR2a)}, i.e. 
\[
\xi_v(l(v \edg w))=\pm l'(f_E(v \edg w))\quad\text{ for all }v\edg w \in E\text{ with }f_{V}(v)\neq f_{V}(w),
\] 
defines an automorphism $(\id_V,\{f_{l,v}:=\xi_v\}_{v\in V})$ of $\Gg$.

\subsection{Quotient graphs}
Given a moment graph $\Gg=\big((V,\le), l\colon E\to \Ln\big)$ we introduce the notion of quotient of $\Gg$ as follows. First, we choose an equivalence relation on $V$ which is compatible
with the structure of a moment graph: 

\begin{dfn}
Let `$\sim$' be an equivalence relation on $V$. We say that `$\sim$' is {\it $\Gg$-compatible} if the following two conditions are satisfied:
\begin{itemize}
\item[(EQV1)] 
$v\sim w$ $\impl$ $v\sim u$ for all $v\le u\le w$,
\item[(EQV2)]
$v_1\to w_1\in E$, $v_1\not\sim w_1$ $\impl$ for any $v_2\in V$, $v_2\sim v_1$ there exists a unique $w_2\in V$ such that $w_2\sim w_1$,   $v_2\to w_2\in E$; \\ moreover, $l(v_1\to w_1)= l(v_2\to w_2)$.
\end{itemize}
\end{dfn}

Then we take a quotient of $\Gg$ with respect to this equivalence relation:

\begin{dfn}
Given a $\Gg$-compatible equivalence relation on $V$ we define a {\it quotient} of $\Gg$ by `$\sim$' denoted $\Gg_\sim=\big((V_{\sim},\le_\sim),l_{\sim}\colon E_{\sim}\to\Ln\big)$ to be the oriented labelled graph where
\begin{itemize}
\item[(Q1)] 
 $V_{\sim}$ is the set of equivalence classes $\{[v]\}_{v\in V}$ of $V$ with respect to `$\sim$';
\item[(Q2)] 
$E_{\sim}:=\{[v]\to [w] \mid v\not\sim w,\,\exists v'\sim v, \,w'\sim w \text{ with } v'\to w'\}$;
\item[(Q3)] 
$\le_{\sim}$ is the transitive closure of relations $[v]\le_{\sim} [w]$, $[v]\to [w]\in E_{\sim}$;
\item[(Q4)] 
$l_{\sim}([v]\to [w]):=l(v'\rightarrow w')$, where $v'\sim v$,  $w'\sim w$ and $v'\to w'\in E$.
\end{itemize}
 Observe 
 that when '$\sim$' is trivial, i.e. $v\sim w$ for all $v$, $w\in V$, then $V_\sim$ on $\La$ consists of one point only, together with the lattice $\La$ ($E_\sim$ is an empty set and there is no label function).
\end{dfn}

Finally, we show that the graph obtained in such a way is a moment graph:

\begin{lem}
Given a moment graph $\Gg$ and a $\Gg$-compatible equivalence relation `$\sim$', the quotient $\Gg_\sim$ is a moment graph on $\La$.
\end{lem}

\begin{proof} 
It reduces to show that $\Gg_\sim$ has no oriented cycles. Indeed, suppose there is an oriented cycle $[v_1]\to [v_2]\to \ldots \to [v_n]\to [v_1]$. Then by \rm{(Q2)} and \rm{(EQV2)} there exists a  path $v_1\to v_2'\to \ldots \to v_n'\to v_1'$ on the graph $\Gg$ for certain $v_i'\sim v_i$. Hence, we obtain $v_1\le v_2'\le \ldots \le v_n'\le v_1'$ with $v_1\sim v_1'$ and by \rm{(EQV1)} it implies $[v_1]= [v_i]$ for all $i$.
\end{proof}

\begin{ex}\label{ex:quotBruhat}
Let $\Gg=\Gg(W)$ be the Bruhat graph and let $W_\Th$ be the parabolic subgroup of $W$ of Example~\ref{ex:Bruhat}. Then the relation on $W$ defined by 
\[
v\sim w \eqst vW_\Th=wW_{\Th}
\] 
is a $\Gg$-compatible equivalence relation and $\Gg_\sim$ can be identified with $\Gg(W^\Th)$, the parabolic Bruhat moment graph from Example~\ref{ex:Bruhat}.
\end{ex}

We will use the following notion introduced in~\cite{B}:
\begin{dfn}\label{ex:specialmatching}
Given a poset $(V,\le)$ denote by `$\trg$' the covering relation on $V$, that is $v\trg w$ if and only if $v\leq w$ and
\[
v\leq u\leq w \impl u=v\text{ or }u=w. 
\]
By a {\it special matching} of $V$ we call a (set) bijection $M\colon V\ra V$ such that 
\begin{itemize}
\item
for any $v\in V$ either $M(v)\trg v$ or $v\trg M(v)$, and  
\item
if $M(v)\neq w$, then $v\trg w$ $\impl$ $M(v)\leq M(w)$.
\end{itemize}
\end{dfn}

\begin{lem}\label{lem:specm}
Given a moment graph $\Gg=\big((V,\le), l\colon E\to \Ln\big)$ assume that the poset $(V,\le)$ admits a special matching such that if $v\to w\in E$, then 
\begin{itemize}
\item[(i)] $M(v)\to M(w)\in E$, and
\item[(ii)] $l(M(v)\to M(w))= l(v\to w)$.
\end{itemize}
Then there is a $\Gg$-compatible equivalence relation on $V$ with equivalence classes given by $[v]=\{v,M(v)\}$, $v\in V$.
\end{lem}

\begin{ex}(cf. \cite[Prop.~4.1]{B})
Suppose $[v,w]$, where $v,w\in W$ is the Bruhat interval which is stable under the right multiplication by a simple reflection~$s$. Then $M(u):=us$ for any $u\in [v,w]$ defines a special matching which satisfies both (i) and (ii) of the lemma.
\end{ex}

\section{Structure algebras of moment graphs}\label{sec:stralg}

In this section we introduce structure algebras associated to the symmetric algebra and the group ring of a lattice $\La$, respectively. We study
its behaviour with respect to characteristic maps and filtrations. We then discuss twisted pull-back maps induced by morphisms of moment graphs and monodromies.

\subsection{Two filtrations}
Consider the two covariant functors 
\[
S^*(-)\colon \La\mapsto S^*(\La)\;\text{ and }\;\Z[-]\colon \La\mapsto  \Z[\La]
\]
from the category of lattices (free finitely generated abelian groups) to the category of commutative rings given by taking the symmetric algebra and the group ring of a lattice $\La$, respectively. By definition, the $i$-th graded component $S^i(\La)$ is additively generated by monomials $\la_1\la_2\ldots \la_i$ with $\la_j \in \La$ and the group ring $ \Z[\La]$ is additively generated by exponents $e^\la$, $\la\in \La$. Let $I_a$ and $I_m$ denote the kernels of the augmentation maps $\ep_a\colon S^*(\La)\to  \Z$ and $\ep_m\colon \Z[\La]\to  \Z$ given by $\la\mapsto 0$. By definition, the ideal $I_a$ consists of polynomials with trivial constant terms and the ideal $I_m$ is generated by differences $(1-e^{-\la})$, $\la\in \La$. Consider the respective $I$-adic filtrations:
\[
S^*(\La)=I_a^0 \supseteq I_a \supseteq I_a^2 \supseteq \ldots  \;\text{ and }\; \Z[\La]=I_m^0 \supseteq I_m \supseteq I_m^2 \supseteq \ldots
\]
Let
\[
gr_a^*(\La)=\bigoplus_{i\ge 0} I_a^{i}/I_a^{i+1} \;\text{ and }\; gr_m^*(\La)=\bigoplus_{i\ge 0}I_m^i/I_m^{i+1}
\]
denote the associated graded rings. Observe that  $gr_a^*(\La)= S^*(\La)$.

\begin{ex}
If $\La\simeq\Z$, then the ring $S^*(\La)$ can be identified with the polynomial ring in one variable $ \Z[x]$, where $x$ is a generator of $\La$. The group ring $ \Z[\La]$ can be identified with the Laurent polynomial ring $\Z[t,t^{-1}]$, where $t=e^x$. The augmentation maps $\ep_a$ and $\ep_m$ are given by
\[
\ep_a\colon x\mapsto 0 \;\text{ and }\; \ep_m\colon t\mapsto 1.
\]
We have $I_a=(x)$ and $I_m$ is additively generated by differences $(1-t^n)$, $n\in\Z$.
\end{ex}

\subsection{Structure algebras and characteristic maps}
We are ready to introduce the following central object of the present paper:

\begin{dfn}
Let $\Gg=\big((V,\le), l\colon E\to \Ln\big)$ be a moment graph on $\La$. Consider the algebras $S=\Z[\La]$ or $S=S^*(\La)$. By a {\it structure algebra} of $\Gg$ we call the $S$-submodule
\[
\SA(\Gg):=\left\{(z_v)_v\in \prod_{v\in V} S\mid z_v-z_w \in x_{l(v\to w)}S,\; \forall \; v\to w\in E\right\}
\] 
with the coordinate-wise multiplication, where $x_\la=1-e^{-\la}$ for $S=\Z[\La]$ and $x_\la=\la$ for $S=S^*(\La)$. In the first case we denote it by $\SA_m(\Gg)$ and in the second case by $\SA_a(\Gg)$. 
\end{dfn}

The grading of $S^*(\La)$ induces a natural grading on the structure algebra $\SA_a(\Gg)$. We denote by $\SA_a^i(\Gg)$ its $i$th-graded homogeneous component. Observe that the algebra $\SA_m(\Gg)$ is not necessarily graded but only filtered. The $I_m$-adic filtration on $\Z[\La]$ induces the (coordinate-wise) filtration on $\SA_m(\Gg)$.

\begin{rem}
The structure algebra $\SA_a(\Gg)$ appears naturally as ring of global sections of the so called structure sheaf on moment graph. It also computes the $T$-equivariant Chow ring/singular cohomology in the case of a flag variety \cite{KK86}. On the other side, the structure algebra $\SA_m(\Gg)$ can be viewed as $K$-theoretic version of $\SA_a(\Gg)$ as it computes the $T$-equivariant $K$-theory of a flag variety \cite{KK90}. We refer to Section~\ref{sec:KacM} for a more detailed description of these structure algebras in the context of equivariant cohomology/$K$-theory.
\end{rem}

For any automorphism of the lattice $\La$ we denote by the same symbol the induced automorphism of $S$.
\begin{dfn}
Given a $\Gg$-monodromy $\xi=\{\xi_v\}_{v\in V}$ the map 
\[
S\ra \prod_{v\in V}S,\quad z\mapsto (\xi_v(z))_{v\in V}
\] 
induces a ring homomorphism $c^\xi\colon S\ra \SA$ called the {\it $\xi$-characteristic map}.
\end{dfn}

\begin{ex} \label{ex:xicharmap} 
If $\xi=\{\id_v\}_{v\in V}$, then the $\xi$-characteristic map is nothing but the structure map
\[
S\ra \SA, \quad z\mapsto (z)_{v\in V}.
\]
Let $\Gg=\Gg(W)$ be a Bruhat moment graph and let $\xi=\{\xi_w\}_{w\in W}$, where $\xi_w$ is the automorphism of the root lattice $\La=\La_r$ given by the $W$-action $\xi_w(\la):=w(\la)$. Then the $\xi$-characteristic map coincides with the characteristic map on structure algebras of \cite[\S4]{Fi08}.
\end{ex}

\subsection{Pullbacks} We extend the notion of a pull-back map for an equivariant cohomology to the setup of structure algebras as follows:

\begin{lem}\label{lem:pullb}
Let $f=(f_V,\{f_{l,v}\}_{v\in V})\colon \Gg\to \Gg'$ be a moment graph morphism and let $\xi=\{\xi_v\}_{v\in V}$ be a $\Gg$-monodromy such that if $v\edg w\in E$ and $f_V(v)\neq f_V(w)$, then $\xi_v(l'(f_E(v\edg w)))\in l(v\edg  w)\Z$. 

Then there is a ring homomorphism between structure algebras given by
\[
f^{\xi\ast}\colon \SA' \ra \SA, \quad (z_{v'})_{v'\in V'}\mapsto (\xi_v(z_{f(v)}))_{v\in V}.
\]
which we call a pull-back map induced by $f$ and twisted by $\xi$.
\end{lem}

\begin{proof} 
Suppose $v\edg w\in E$ and $f_V(v)=f_V(w)$. Then $z_{f_V(v)}=z_{f_V(w)}=z$ and by~\eqref{eq:mon} we obtain
\[
\xi_v(z)-\xi_w(z)\in l(v\to w)S.
\]
Suppose $v\edg w\in E$ and $f_V(v)\neq f_V(w)$. Then by \rm{(MR1)} $f_V(v) \edg f_V(w) \in E'$ and
\[
z_{f_V(v)}-z_{f_V(w)}\in l'(f_E(v\edg w))S.
\]
So there exists a $g\in S$ such that $z_{f_V(v)}=z+gl'(f_E(v\edg w))$, $z=z_{f_V(w)}$, and hence by~\eqref{eq:mon} and by the hypothesis on $\xi_v$ we obtain
\begin{align*}
\xi_v(z_{f_V(v)})-\xi_w(z_{f_V(w)})&=\xi_v\big(z+gl'(f_E(v\edg w))\big)-\xi_w(z)\\
&=\xi_v(z)-\xi_w(z)+\xi_v(g l'(f_E(v\edg w))) \in l(v\edg w)S.\qedhere
\end{align*}
\end{proof}

\begin{ex} 
Take the Bruhat moment graph $\Gg=\Gg(W)$ and $\Gg'=\Gg(W^\Th)=\Gg_\sim$, where $\sim$ is the $\Gg$-compatible relation of Example~\ref{ex:quotBruhat}. Take the trivial monodromy $\xi=\{\id_w\}_{w\in W}$. Then $f^{\xi\ast}\colon \SA' \to \SA$ is the ring homomorphism sending the element $(z_v')_{v'\in W^\Th}$ to the element $(z_v)_{v\in W}$ with $z_v=z_{v'}$ if $v\in v'W_\Th$. 
By the results of \cite{KK86, KK90} this map coincides with the pull-back on the $T$-equivariant Chow ring/cohomology (resp. $K$-theory) induced by the usual projection $G/B\to G/P_{\Th}$.
\end{ex}


\section{Push-forwards on structure algebras}\label{sec:pushpull}
In the present section we introduce and study push-forwards on structure algebras induced by fibrations of moment graphs. 
First, we introduce the notion of a fibration of moment graphs. Then, twisting it by a monodromy $\xi$ we obtain the so called $\xi$-fibration.
Given a regular $\xi$-fibration we construct the push-forward map and prove the projection formula. 
Finally, for fibrations associated to special matchings, we produce analogues of push-pull operators on moment graphs. 

\subsection{Fibrations}
Given a moment graph $\Gg=\big((V,\le), l\colon E\to \Ln\big)$ and a $\Gg$-compatible equivalence relation `$\sim$' consider the morphism of (oriented labelled) graphs $\pi\colon \Gg\to \Gg_\sim$ induced by the quotient set map $v\to [v]$. 

\begin{dfn}\label{dfn:fibration}
Consider the moment graph obtained as a full subgraph of $\Gg$ by restricting the vertex set to the equivalence class $[v]$. We call it the {\it fibre} of $\pi$ at $[v]$ and denote it by $\Gg_{[v]}$. 

We say that the quotient morphism $\pi$ is a {\it fibration} if for any $[v],[w]\in V_\sim$ there is a moment graph isomorphism 
\[
f^{[v],[w]}=(f_V^{[v],[w]}, \{f^{[v],[w]}_{l,y}\})_{y\in [v]} \colon \Gg_{[v]} \stackrel{\simeq}\ra \Gg_{[w]}\quad\text{such that}
\] 
\begin{itemize}
\item[(FB1)]
for all $u,v,w\in V$, we have $f^{[v],[v]}=\textrm{Id}_{\Gg_{[v]}}$ and $f^{[u],[v]}\circ f^{[v],[w]}=f^{[u],[w]}$;
\item[(FB2)] for any $w\in V$, 
for all $y,y'\in[v]$ such that $y\edg y'\in E$, we have $f^{[v],[w]}_l:=f^{[v],[w]}_{l,y}=f^{[v],[w]}_{l,y'}$ and $f^{[v],[w]}_l(l(y\edg y'))=\pm l(f^{[v],[w]}_V(y)\edg f^{[v],[w]}_V(y'))$.
\end{itemize}
\end{dfn}

\begin{ex}\label{ex:fibbr}
If $\Gg=\Gg(W)$ and $\Gg_\sim=\Gg(W^{\Th})$ as in Example~\ref{ex:quotBruhat}, then $\pi$ is a fibration with isomorphisms  $f^{[v],[w]}\colon \Gg_{[v]}\to \Gg_{[w]}$ given by
\[
f^{[v],[w]}_V\colon vW_\Th\rightarrow wW_\Th, \; u\mapsto \overline{w}\, \overline{v}^{-1}u,\; \text{ and }\; f^{[v],[w]}_l: \la\mapsto \overline{w}\,\overline{v}^{-1}(\la), \la\in \La,
\]
where $\overline{z}$ denotes the minimal lenght representative of the coset $zW_\Th$.
Notice that all the fibres are isomorphic to the Bruhat graph $\Gg(W_\Th)$ which can be identified with the fibre $\Gg_{[e]}$ over the neutral element.
\end{ex}

\begin{ex}\label{ex:specialmatching2}
Assume that the vertex poset $(V,\le)$ of $\Gg$ admits a special matching $M\colon V\to V$ of Definition~\ref{ex:specialmatching}. Then by Lemma~\ref{lem:specm} for any $v\in V$ the fibre  $\Gg_{[v]}$ is a moment graph consisting of two vertices and one labelled arrow $E_{[v]}=\{v \edg M(v)\}$ oriented according to the partial order. As a lattice for $\Gg_{[v]}$ we may take the rank one lattice $\La_{[v]}$ generated by the label $l(v\edg M(v))$. There are obvious isomorphisms of moment graphs $f^{[v],[w]}$ such that $f_V^{[v],[w]}(v)=w$ or $M(w)$ and $f^{[v],[w]}_l(l(v\edg M(v)))=\pm l(w\edg M(w))$ which define a fibration.
\end{ex}

\subsection{Fibers and monodromy}

Given $v\in V$ consider the multi-set of labels $L_v$ of edges adjacent to $v$, i.e.
\[
L_v=\{l(v \edg w) \mid v\edg w\in E\}.
\]
We denote by $L_{v,\sim}$ its subset entirely contained in the fibre $\Gg_{[v]}$, i.e.
\[
L_{v,\sim}=\{l(v \edg w) \mid v\edg w\in E,\; v\sim w\}.
\]

We denote by $-L_{v,\sim}$ the multi-set $\{-l(v \edg w) \mid v\edg w\in E,\; v\sim w\}$ respectively. 
We assume that all moment graphs are {\it locally finite}, that is for any $v$, $w\in V$ the interval $\{u\in V\mid v\leq u\leq w \}$ is finite.

\begin{dfn}\label{dfn:fibcompatible} 
Let $\pi\colon \Gg\to \Gg_\sim$ be a fibration with isomorphisms $(f^{[v],[w]})_{(v,w)\in V\times V}$ between the fibres as in Definition~\ref{dfn:fibration}. Let $[e]\in V_\sim$ be a distinguished vertex of the quotient graph. A collection $\xi=\{\xi_y\}_{y\in V}$ of automorphisms of $\La$ is said to be compatible with the fibre $\Gg_{[e]}$ if  
\begin{itemize}
\item[(CF1)] 
For any $y \in [v]$ we have 
\[
\xi_y\circ f^{[v],[e]}_l(\prod_{\ga\in L_{y,\sim}}\ga)=\pm\prod_{\ga\in L_{y,\sim}}\ga,
\] 
where the products are in $S^*(\La)$. In particular, for any  $\ga\in L_{y,\sim}$ we have \[\xi_y(f^{[v],[e]}_l(\ga))\in \pm L_{y,\sim}.\]
\item[(CF2)] 
For any $y\in [v]$ set 
\[N_y^{\xi}=\{f^{[v],[e]}_l(\ga)\mid \ga\in L_{y,\sim}\hbox{ and } \xi_y(f^{[v],[e]}_l(\ga))\in -L_{y,\sim}\}.
\]
Then for any $y\to y' \in E$ with $y'\in [v]$ the following two conditions hold:
\begin{itemize}
\item[(a)] $\# N_y^{\xi}\not\equiv \#N_{y'}^{\xi}\mod 2$;
\item[(b)] $\sum_{\be\in N_y^{\xi}}\xi_y(\be)-\sum_{\de\in N_{y'}^{\xi}}\xi_{y'}(\de)\in l(y\to y')\Z$.
\end{itemize}
\end{itemize}
\end{dfn}

\begin{ex}\label{ex:gcomf}
In the setup of Examples~\ref{ex:Bruhat} and \ref{ex:quotBruhat}, if $v\in W^\Th$ is a minimal length representative and $e$ is the neutral element of $W$, then for any $y\in [v]$ we have $v(L_{e,\sim})=L_{y,\sim}$, where $L_{e,\sim}$ is the set of positive roots $\Phi_+^\Th$ of the root subsystem of $\Phi$ spanned by simple roots from $\Th$. So $L_{y,\sim}$ does not depend on a choice of the representative so we can simply denote it by $L_{[v]}$.

Suppose $y\in [v]$ so that by the parabolic decomposition $y=vu$, where $v\in W^\Th$ and $u\in W_\Th$. 

Set $\xi_y(\la):=y(\la)$ and $f_l^{[v],[e]}(\la):=v^{-1}(\la)$, $\la\in\La$ as in Example~\ref{ex:fibbr}. Then we obtain 
\[
\xi_y\big( f^{[v],[e]}_l(\prod_{\ga\in L_{[v]}}\ga)\big)=vu( \prod_{\be\in L_{[e]}}\be).
\]
Denote by $\prod \Phi_+^\Th$ the product $\prod_{\be\in \Phi_+^\Th}\be$. Since $u(\prod \Phi_+^\Th)=(-1)^{\ell(u)}(\prod \Phi_+^\Th)$ for any $u\in W_\Th$, the property (CF1) follows.

By definition we have
\begin{align*}
N_y^\xi &=\{v^{-1}(\ga)\mid \ga\in v(\Phi_+^\Th)\text{ and }yv^{-1}(\ga)\in v(\Phi_-^\Th) \}\\
&=\{\be\mid \be\in \Phi_+^\Th \text{ and }u(\be)\in \Phi_-^\Th \},
\end{align*}
therefore, the cardinality $\#N_y^\xi$ coincides with the length $\ell(u)$ of $u$.

Suppose that $y\to y'\in E$ and $y'\in[v]$, then $y=s_\al y'$ and $y'=vu'$, $u'\in W_\Th$. 

Therefore, $vu=s_\al vu'$ which implies that $u'=s_{v^{-1}(\al)}u$ and, hence, $\ell(u)\not\equiv\ell(u')\mod 2$. This verifies property (CF2a).

As for (CF2b), observe that there is a bijection (c.f.~\cite[Proposition 3.2.14]{CS})
\[
N_{y}^\xi\setminus N_{s_{y^{-1}(\al)}}^\xi\stackrel{\sim}{\ra}N_{y'}^\xi\setminus N_{s_{y^{-1}(\al)}}^\xi, \qquad \be\mapsto s_{y^{-1}(\al)}(\be),
\]
moreover, $N_{y}^\xi\cap N_{s_{y^{-1}(\al)}}^\xi\subset N_{y'}^\xi\cap N_{s_{y^{-1}(\al)}}^\xi$ and 
\[
\be\in (N_{y'}^\xi\cap N^\xi_{s_{y^{-1}(\al)}})\setminus N_y^\xi\quad\Leftrightarrow\quad -s_{y^{-1}(\al)}(\be)\in (N_{y'}^\xi\cap N^\xi_{s_{y^{-1}(\al)}})\setminus N_y^\xi.
\]
Therefore, we obtain
\begin{align*}
\sum_{\be\in N_y^\xi}y(\be)-\sum_{\de\in N_{y'}^\xi}y'(\de)&=\sum_{\be \in N_{y}^\xi\setminus N_{s_{y^{-1}(\al)}}^\xi}y(\be)-y'(s_{y^{-1}(\al)}(\be))\\
&\quad + \sum_{\be\in N_{y}^\xi\cap N_{s_{y^{-1}(\al)}}^\xi}y(\be)-y'(\be)\\
&\quad -\frac{1}{2}\sum_{\be\in (N_{y'}^\xi\cap N^\xi_{s_{y^{-1}(\al)}})\setminus N_y^\xi}y'(\be+(-s_{y^{-1}(\al)}(\be)))\\
&\in \mathbb{Z}\al.
\end{align*}
\end{ex}

\begin{ex}
In the fibration of Example~\ref{ex:specialmatching2} choose $e\in V$ and let $\Gg_{[e]}$ be the respective fibre. 
Then a $\Gg$-monodromy  $\xi=\{\xi_y\}_{y\in V}$ of Definition~\ref{dfn:gmond} is compatible with the fibre $\Gg_{[e]}$ if and only if 
\begin{equation}\label{eqn:specialmatching3}
\xi_y\big(l(e \edg M(e))\big)=-\xi_{M(y)}\big(l(e \edg M(e))\big)=\pm l(y \edg M(y)).
\end{equation}
Indeed, the properties (CF1), (CF2a) and (CF2b) follow from
\begin{itemize}
\item $f^{[y][e]}\big(l(y \edg M(y))\big)=\pm l(y \edg M(y))$,
\item $L_{[y]}=\{l(y \edg M(y))\}$ for any $y\in V$,
\item $
N_{y}^\xi=\left\{
\begin{array}{ll}
l(y \edg M(y))&\text{ if }\xi_{y}\big(l(e \edg M(e))\big)=- l(y\edg M(y)),\\
\emptyset&\text{ otherwise}.
\end{array}
\right.
$\end{itemize} 
\end{ex}

\subsection{Push-forwards}

Let $\Gg=\big((V,\le), l\colon E\to \Ln\big)$ be a moment graph together with a $\Gg$-compatible equivalence relation `$\sim$', a $\Gg$-monodromy $\xi=\{\xi_v\}_{v\in V}$ and a distinguished vertex $e\in V$. We now introduce the notion of a $\xi$-fibration between moment graphs.
\begin{dfn}
The induced quotient map $\Gg\to \Gg_\sim$ is called a {\em $\xi$-fibration} and denoted $\pi^\xi$ if it is a fibration such that the monodromy $\xi$ is compatible with the fibre $\Gg_{[e]}$. 
\end{dfn}

\begin{dfn} \label{dfn:regf}
We say that a $\xi$-fibration $\pi^\xi$ is regular if 
\begin{itemize}
\item
$L_{v,\sim}=L_{w,\sim}$ for any $v \sim w$ in $V$ (observe that this implies that $L_{v,\sim}$ does not depend on a choice of a representative so it can be denoted by $L_{[v]}$) 
\end{itemize}
and for each $[v]\in V_\sim$ we have 
\begin{itemize}
\item
$x_\ga=1-e^{-\ga}$ is irreducible in $\Z[\La]$ for each $\ga\in L_{[v]}$, and
\item
$x_\ga \mid x_{\ga'}x$, $x\in \Z[\La]$,  $\ga,\ga'\in L_{[v]}$, $\ga\neq \ga'$ $\impl$ $x_\ga\mid x$.
\end{itemize}
\end{dfn}

\begin{prop} \label{prop:push} 
Given a regular $\xi$-fibration $\pi^\xi\colon \Gg \to \Gg_\sim$, there is a $\Z[\La]$-module homomorphism between the associated (multiplicative) structure algebras
\[
\pi_{*}^\xi\colon \SA_m(\Gg)\ra \SA_m(\Gg_\sim) \;\text{ defined by }\; (z_y)_{y\in V}\mapsto \Big(\sum_{y\in[v]}\frac{z_y}{\xi_y\big(\prod_{\be\in L_{[e]}} x_{\be}\big)}
\Big)_{[v]\in V_\sim},
\] 
where $z_y\in \Z[\La]$ and $x_\be=1-e^{-\be} \in \Z[\La]$.
\end{prop}

\begin{proof}
We have to show that 
\begin{enumerate}
\item $\tilde{z}_{[v]}:=\sum_{y\in[v]}\frac{z_y}{\xi_y(\prod_{\be\in L_{[e]}}x_\be)}\in \Z[\La]$, and
\item $\tilde{z}_{[v]}-\tilde{z}_{[w]}\in x_{l([v]\to [w])}\Z[\La]$ for any $[v]\to [w]\in E_\sim$.
\end{enumerate}

(1): Set $\be=f_l^{[v],[e]}(\ga)\in L_{{e}}$ for $\ga \in L_{[v]}$. Then
\[
N_y^\xi=\{\be\in L_{[e]}\mid \xi_y(\be)\in -L_{[v]}\}.
\]
Since $x_{-\ga}=x_\ga(-e^\ga)$, by Definitions~\ref{dfn:fibration} and~\ref{dfn:fibcompatible} we obtain 
\[
\xi_y(\prod_{\be\in L_{[e]}}x_\be)=\Big(\prod_{\ga\in L_{[v]}}x_\ga\Big)\prod_{\be\in N_y^\xi}(-e^{-\xi_y(\be)}).
\]
Therefore, it reduces to check that
\[
\tfrac{1}{\prod_{\ga\in L_{[v]}}x_\ga}\sum_{y\in [v]}\Big(\sgn(y)z_y \prod_{\be\in N_y^\xi}e^{-\xi_y(\be)}\Big)\in \Z[\La],
\]
where we set $\sgn(y):=(-1)^{\#N_y^\xi}$. Observe also that Definition~\ref{dfn:fibcompatible} implies that if $y,y'\in[v]$ and $y\to y'\in E$ then $\sgn(y)=-\sgn(y')$.

Since the fibration is regular, this is equivalent to 
\[
x_\ga\mid \sum_{y\in [v]}\Big(\sgn(y)z_y \prod_{\be\in N_y^\xi}e^{-\xi_y(\be)}\Big), \quad \text{ for any }\ga\in L_{[v]}.
\]

Each label $\ga\in L_{[v]}$ induces an automorphism $j_\ga$ on the set of vertices of $\Gg_{[v]}$ by mapping $y\in [v]$ to the unique vertex $j_\ga(y)\in [v]$ connected to $y$ via an edge labelled by $\ga$. So we obtain
\begin{align*}
&\sum_{y\in [v]}\Big(\sgn(y)z_y \prod_{\be\in N_y^\xi}e^{-\xi_y(\be)}\Big)\\ &=\sum_{\substack{y\in [v]\\
y<j_\ga(y)}}\sgn(y)\Big(z_y \prod_{\be\in N_y^\xi}e^{-\xi_y(\be)}-z_{j_\ga(y)} \prod_{\de\in N_{j_\ga(y)}^\xi}e^{-\xi_{j_\ga(y)}(\de)}\Big)\\
&\equiv\sum_{\substack{y\in [v]\\
y<j_\ga(y)}}\sgn(y)z_y\Big(\prod_{\be\in N_y^\xi}e^{-\xi_y(\be)}-\prod_{\de\in N_{j_\ga(y)}^\xi}e^{-\xi_{j_\ga(y)}(\de)}\Big)\mod x_\ga\\
&=\sum_{\substack{y\in [v]\\
y<j_\ga(y)}}\sgn(y)z_y\prod_{\be\in N_y^\xi}e^{-\xi_y(\be)}\Big(1-e^{-\sum_{\de\in N_{j_\ga(y)}^\xi}\xi_{j_\ga(y)}(\de)+\sum_{\be\in N_y^\xi}\xi_y(\be)}\Big).
\end{align*}
By (CF2b), there exists an integer $m\in \Z$ such that
\[
\Big(1-e^{-\sum_{\de\in N_{j_\ga(y)}^\xi}\xi_{j_\ga(y)}(\de)+\sum_{\be\in N_y^\xi}\xi_y(\be)}\Big)=1-e^{-m\ga},
\]
and, therefore, $x_\ga=1-e^{-\ga}$ divides the latter.

(2): Assume now that $[v]\to [w]\in E_\sim$ with $l_\sim([v]\to [y])=\al$ and consider 
\[
\tilde{z}_{[v]}-\tilde{z}_{[w]}=\sum_{y\in [v]}\tfrac{z_y}{\xi_y(\prod_{\be\in L_{[e]}}x_\be)}-\sum_{u\in [w]}\tfrac{z_u}{\xi_u(\prod_{\be\in L_{[e]}}x_\be)}.
\]
Since `$\sim$' is $\Gg$-compatible, there is a similar bijection $j_\al$ between the vertices of $\Gg_{[v]}$ and $\Gg_{[w]}$. We now have to show that
\[
\sum_{y\in [v]}\Big(\tfrac{z_y}{\xi_y(\prod_{\be\in L_{[e]}}x_\be)}-\tfrac{z_{j_\al(y)}}{\xi_{j_\al(y)}(\prod_{\be\in L_{[e]}}x_\be)}\Big)\in x_\al \Z[\La].
\]
The independence of labels implies that $\al\not\in L_{[y]}\cup L_{[j_\al(y)]}$. Hence, we are reduced to show that
\[
z_y\xi_{j_\al(y)}(\prod_{\be\in L_{[e]}}x_\be)-z_{j_\al(u)}\xi_{y}(\prod_{\be\in L_{[e]}}x_\be)\in x_\al \Z[\La].
\]
The latter follows since  $\xi_{j_\al(y)}(\prod_{\be\in L_{[e]}}x_\be)-\xi_{y}(\prod_{\be\in L_{[e]}}x_\be)\in x_\al \Z[\La]$ by Definition~\ref{dfn:gmond}, and $z_y-z_{j_\al(y)}\in x_\al \Z[\La]$ as $z\in \SA_m(\Gg)$.
\end{proof}

Replacing $x_\ga$ by $\ga$ and $\Z[\La]$ by $S^*(\La)$ in Definition~\ref{dfn:regf} and in Proposition~\ref{prop:push} the same proof (where $e^\ga$ is replaced by $1$) gives the similar result for the structure algebras $\SA_a$ associated to $S^*(\La)$:

\begin{cor}
Given a regular $\xi$-fibration $\pi^\xi\colon \Gg \to \Gg_\sim$, there is a $S^*(\La)$-module homomorphism between the associated (additive) structure algebras
\[
\pi_{*}^\xi\colon \SA_a(\Gg)\ra \SA_a(\Gg_\sim) \;\text{ defined by }\; (z_y)_{y\in V}\mapsto \Big(\sum_{y\in[v]}\frac{z_y}{\xi_y\big(\prod_{\be\in L_{[e]}} \be\big)}
\Big)_{[v]\in V_\sim},
\] 
where $z_y\in S^*(\La)$.
\end{cor}

\begin{cor}[Projection formula]\label{cor:projform}
In the hypothesis of Proposition~\ref{prop:push} we have 
\[
\pi_\ast^\xi((\pi^{\id_\La})^\ast(z')\cdot z)=z'\cdot \pi_\ast^\xi(z),\quad z'\in \SA_m(\Gg_\sim),\; z\in \SA_m(\Gg).
\]
In other words, the push-forward $\pi_\ast^\xi$ is a homomorphism of $\SA_m(\Gg_\sim)$-modules.
\end{cor}

\begin{proof} 
Let $z'=(z'_{[v]})_{[v]\in V_\sim}$ and $z=(z_v)_{v\in V}$. By definition, we have
\begin{align*}
\pi_\ast^\xi((\pi^{\id_\La})^\ast(z')\cdot z)&=\pi_\ast^\xi((z'_{[v]}z_v)_{v\in V})=\Big(\sum_{y\in[v]}\frac{z'_{[v]}z_y}{\xi_y\big(\prod_{\be\in L_{[e]}} x_{\be}\big)}\Big)_{[v]\in V_\sim} \\
&=\Big(z'_{[v]}\sum_{y\in[v]}\frac{z_y}{\xi_y\big(\prod_{\be\in L_{[e]}} x_{\be}\big)}\Big)_{[v]\in V_\sim}=(z'_{[v]} \tilde{z}_{[v]})_{[v]\in V_\sim}. \qedhere
\end{align*}
\end{proof}

\begin{ex}\label{ex:fibK}
Consider the fibration $\pi\colon \Gg\to \Gg_\sim$ of Example~\ref{ex:gcomf}, where $\Gg=\Gg(W)$ is the Bruhat graph, $\Gg_\sim=\Gg(W^\Th)$ is the parabolic Bruhat moment graph, $e$ is the neutral element of $W$, the moment graph isomorphism $f^{[v],[e]}\colon \Gg_{[v]}\to\Gg_{[e]}$ are given by $f^{[v],[e]}_V(w)=v^{-1}w$ for any $w\in vW_\Th$ and $f^{[v],[e]}_l(\la)=v^{-1}(\la)$, $\la\in \La$, the automorphisms $\xi=\{\xi_y\}_y$ are given by $\xi_y=y(\la)$, $\la\in \La$. 

Then we are in the hypotheses of the proposition and the induced homomorphism of structure algebras coincides with the classical push-forward map on $K$-theory $\pi_\ast\colon K(G/B) \to K(G/P_\Th)$ induced by the canonical quotient map $G/B\to G/P_\Th$ (see e.g. \cite[(3.3)]{GR}).
\end{ex}

\begin{ex} 
Consider the fibration of Example~\ref{ex:specialmatching2}. Suppose that the multi-sets $L_{[v]}$ for $\Gg_\sim$ consist of linearly independent labels. Suppose also that we have a collection $\xi=\{\xi_y\}$ of automorphisms of $\La$ satisfying~\eqref{eqn:specialmatching3}. Then we are in the hypothesis of the proposition and there is the induced homomorphism of $\Z[\La]$-modules $\pi_\ast^\xi\colon \SA_m(\Gg) \to \SA_m(\Gg_\sim)$.

Combining it with the induced pull-back we obtain the group homomorphism
\[
(\pi^{\id_\La})^\ast \circ \pi^\xi_\ast \colon \SA_m(\Gg) \ra \SA_m(\Gg)
\]
which we call the \emph{push-pull operator} on the moment graph $\Gg$.

In the case of a special matching of the Weyl group induced by right multiplication by a simple reflection $s$ and $\xi$ taken from Example~\ref{ex:xicharmap}, the above composition will give the classical divided difference operator. 
\end{ex}


\section{The Chern character and the Riemann-Roch Theorem}\label{sec:chern}

In the present section we introduce the Chern character between structure algebras of moment graphs. We study its properties with respect to characteristic maps, pull-backs and forgetful maps. We state and prove our main result -- the analogue of the Riemann-Roch theorem for moment graphs.

\subsection{Truncated Chern character}
Consider a map
\[
\ch_i\colon \Z[\La] \to S^{\le i}_\Q(\La):=S^*(\La)/I_a^{i+1}\otimes_{\Z}\Q
\]
defined by taking the truncated exponential series
\[
\ch_i(e^{\la}) \mapsto \exp(\la)=\sum_{0\le j\le i}\tfrac{1}{j!}\la^j, \quad \la \in \La.
\]
It is a ring homomorphism which we  call a (truncated) Chern character.

Observe that under this map 
\[
\ch_i(x_\la)=ch_i(1-e^{-\la})=\sum_{1\le j\le i}\tfrac{(-1)^{j+1}}{j!}\la^j \in \la S^{\le i}_\Q(\La)
\] 
so $\ch_i(I_m^j) \subset I_a^j$ for all $j\le i$ and, therefore, there is an induced graded ring homomorphism 
\[
\ch_i\colon gr_m^{\le i}(\La) \to gr_a^{\le i}(\La)_\Q
\]
which becomes an isomorphism after tensoring the left hand side with $\Q$. 

Since $\ch_i(x_\la)\in \la S^{\le i}_\Q(\La)$, $ch_i$ preserves the relations in the definition of the structure algebra. So we obtain

\begin{prop}\label{thm:main}
Let $\Gg=\big((V,\le), l\colon E\to \Ln\big)$ be a moment graph. Then the direct sum 
\[
\oplus_{v\in V} \ch_i \colon \oplus_{v\in V}\Z[\La] \to\oplus_{v\in V}S^{\le i}_\Q(\La)
\]
of the maps $\ch_i$ restricts to a ring homomorphism (called the localized Chern character)
\[
\cch_i\colon \SA_{m}(\Gg) \to \SA_{a}^{\le i}(\Gg)_\Q,
\]
where the latter is the truncated structure algebra with $\Q$-coefficients.
\end{prop}

\subsection{Forgetful, characteristic maps and pull-backs}
We now show that the localized Chern character commutes with  characteristic, forgetful maps and pull-backs. Our first observation is the following

\begin{lem}\label{lem:charm}
The localized Chern character on the structure algebras of $\Gg$ respects the $\xi$-characteristic map, i.e.
\[
c^\xi \circ \ch_i =\cch_i\circ c^\xi \colon S_m \to \SA_a^{\le i}(\Gg)_\Q.
\] 
\end{lem}

\begin{proof}
As both functor $S^*(-)$ and $\Z[-]$ and $\ch_i$ are functorial with respect to automorphisms $\xi_x$ of the lattice $\La$, the lemma follows.
\end{proof}

As in \cite{LZ19} we denote by $\widetilde{\SA}(\Gg)$ the quotient of $\SA(\Gg)$ modulo the ideal $I\SA(\Gg)$ (here $\SA(\Gg)$ is viewed as a $S$-module) and call it the augmented structure algebra. The quotient  map $\rho \colon \SA(\Gg) \to \widetilde{\SA}(\Gg)$ is called the forgetful map. We have

\begin{lem}\label{lem:forgt} 
The localized Chern character on structure algebras respects the forgetful map, i.e.
\[
\rho \circ \cch_i=\widetilde{\cch}_i\circ \rho \colon \SA_m(\Gg) \to \widetilde{\SA}_a^{\le i}(\Gg)_\Q,
\]
where $\widetilde{\cch}_i$ denotes the restricted class.
\end{lem}

\begin{proof}
Since the Chern character preserves the augmentation ideal, the result follows.
\end{proof}

Finally, let $f\colon \Gg\to \Gg'$ be a morphism of oriented graphs which satisfies the hypothesis of Lemma~\ref{lem:pullb}. Consider the induced pull-back map on the structure algebras $f^{\xi\ast}\colon \SA(\Gg')\to\SA(\Gg)$. By definition, we then have

\begin{lem}\label{lem:cpull}
The localized Chern character respects the pull-backs, i.e. we have
\[
f^{\xi\ast}\circ \cch_i' = \cch_i \circ f^{\xi\ast},
\]
where $\cch_i'$ is the respective Chern character for the moment graph $\Gg'$.
\end{lem}

\subsection{The Riemann-Roch type theorem}

Suppose now we are in the hypotheses of Proposition~\ref{prop:push}, i.e. we are given a regular $\xi$-fibration $\pi^\xi\colon \Gg \to \Gg_\sim$ of moment graphs with a distinguished point $e\in V$. So that there is the induced push-forward $\pi_\ast^\xi\colon \SA(\Gg) \to \SA(\Gg_\sim)$. We then define the $\xi$-Todd genus of the fibration $\pi^\xi$ to be the truncation
\[
td^\xi_i(\Gg)=\big(\exp\big(\sum_{\be\in N_y^\xi}-\xi_y(\be)\big)\big)_{\! y} \in \SA_a^{\le i}(\Gg).
\]

We have the following Riemann-Roch type theorem for $\xi$-fibrations on moment graphs:
\begin{thm}\label{thm:mainRR}
For any $z\in \SA_m(\Gg)$ we have
\[
\pi^\xi_\ast \big(\cch_i(z)\cdot td_i^\xi(\Gg)\big)=\cch_i(\pi^\xi_\ast(z)).
\] 
\end{thm}

\begin{proof}
It is enough to prove it on fibres $\Gg_{[v]}$, $[v]\in V_\sim$. Recall that the map $\pi^\xi_\ast$ is given on the fibre by 
\[
(z_y)_{y\in [v]} \mapsto \tilde z_{[v]}=\tfrac{1}{\prod_{\ga\in L_{[v]}}x_\ga}\sum_{y\in [v]}\Big(\sgn(y)z_y \prod_{\be\in N_y^\xi}e^{-\xi_y(\be)}\Big)
\]
Since $\cch_i$ is a ring homomorphism and $\cch_i(x_\ga)=\ga$, we obtain
\begin{align*}
\cch_i(\pi^\xi_\ast(z))_{[v]} 
&=  \tfrac{1}{\prod_{\ga\in L_{[v]}} \ga}\sum_{y\in [v]}\Big(\sgn(y)\cch_i(z_y) \cch_i\big(\prod_{\be\in N_y^\xi}e^{-\xi_y(\be)}\big)\Big)\\
&=\tfrac{1}{\prod_{\ga\in L_{[v]}} \ga}\sum_{y\in [v]}\Big(\sgn(y)\cch_i(z_y) \exp\big(\sum_{\be\in N_y^\xi}-\xi_y(\be)\big)\Big)\in S^*(\La).
\end{align*}
On the other side we obtain
\[
\pi^\xi_\ast \big(\cch_i(z)\big)_{[v]}=\tfrac{1}{\prod_{\ga\in L_{[v]}} \ga}\sum_{y\in [v]}\Big(\sgn(y)\cch_i(z_y) \Big)\in S^*(\La).
\]
The result then follows.
\end{proof}

\begin{rem}
The formula of Theorem~\ref{thm:mainRR} can be viewed as the moment graph analogue of~\cite[Corollary~2.5.5]{Pa04} where the $\xi$-Todd genus of the $\xi$-fibration corresponds to $td_{\cch}(T_X)$ of loc.cit.
\end{rem}

\section{The case of Kac-Moody flag varieties}\label{sec:KacM}

In the present section we show that all Examples~\ref{ex:Bruhat}, \ref{ex:quotBruhat}, \ref{ex:xicharmap}, \ref{ex:fibbr}, \ref{ex:gcomf} and, finally, \ref{ex:fibK} can be extended to the Kac-Moody settings (i.e., to possibly infinite Weyl groups), so that our main result Theorem~\ref{thm:mainRR} holds for the $T$-equivariant $K$-theory and Chow groups of Kac-Moody flag varieties. 

\subsection{Preliminaries and notation}
We introduce notation and list basic properties of the root datum associated to a generalized Cartan matrix. We follow closely \cite[(1.1) and (1.2)]{KK90}:

Let $A:=(a_{ij})_{1\le i,j\le l}$ be a generalized Cartan matrix (i.e., $a_{ii}=2$, $-a_{ij}\in \Z_+$ for all $i\neq j$, where $\Z_+$ is the set of nonnegative integers, and $a_{ij}=0$ $\Leftrightarrow$ $a_{ji}=0$). Choose a triple $(\Crt,\Pi,\Pi^\vee)$, unique up to isomorphism, where $\Crt$ is a complex vector space of dimension $2l-\mathrm{rk} A$, 
\[
 \Pi=\{\al_1,\ldots,\al_l\}\subset \Crt^\ast, \quad \Pi^\vee=\{h_1, \ldots,h_l\}\subset \Crt
\] 
are linearly independent sets satisfying $\al_j(h_i)=a_{ij}$. Such a triple we call the root datum corresponding to the generalized Cartan matrix $A$.
 
Let $\Lie=\Lie(A)$ be the Kac-Moody Lie algebra associated to $A$ as in \cite[\S1]{KK90}. So that $\Crt$ is the Cartan subalgebra of $\Lie$ and there is the root space decomposition
\[
\Lie=\Crt \oplus \sum_{\al\in \De_+}(\Lie_\al\oplus \Lie_{-\al}),\quad \Lie_\al=\{x\in \Lie\mid [h,x]=\al(h)x,\;\forall h\in \Crt\},
\]
where $\De_+=\{\al\in \sum_{i=1}^l \Z_+ \al_i\mid \al\neq 0\text{ and }\Lie_\al\neq 0\}$ is called the set of positive roots. Define $\De_-=-\De_+$ and call it the set of negative roots. Define $\De=\De_+\cup \De_-$ and call it the set of roots. The roots $\{\al_i\}_{1\le i\le l}$ are called the simple roots and the elements $h_i$, $1\le i\le l$ are called the simple coroots.

Associated to $(\Lie,\Crt)$ there is the Weyl group $W\subset \mathrm{Aut}(\Crt^\ast)$, generated by the simple reflections $s_i$, $1\le i\le l$, where
\[
s_i(\la)=\la - \la(h_i)\al_i,\quad \la\in \Crt^\ast.
\]
The group $W$ is the Coxeter group on generators $s_i$, $1\le i\le l$. For  $i\neq j$, the order $m_{ij}$ of $s_is_j$ equals to $2,3,4,6,\infty$ when $a_{ij}a_{ji}$ is $0,1,2,3,\ge 4$, respectively. We denote by $\leq$ the Bruhat order on $W$ and by $\ell:W\rightarrow\Z_+$ the length function. The Weyl group preserves $\De$.

Define  the subset of real roots to be 
\[
\Phi:=\{w (\al_i)\mid w\in W,\;  \al_i\in \Pi\}.
\] 
For any $\al=w(\al_i)\in \Phi$, the associated the reflection is $s_\al=ws_iw^{-1}$. Set $\Phi_+=\De_+\cap \Phi$ and $\Phi_-=\De_-\cap \Phi$.

For any $\Th \subset \Pi$ let $W_\Th$ be the subgroup of $W$ generated by $\{s_i\}_{\al_i\in \Th}$. Let $W^\Th$ denote the subset of minimal left coset representatives of $W/W_\Th$ (each coset contains a unique such representative).

As in \cite[(1.2)]{KK90} we fix an integral lattice $\Crt_{\Z}\subset \Crt$ satisfying
\begin{itemize}
\item $h_i\in \Crt_{\Z}$ for all $1\le i\le l$,
\item $\Crt_{\Z}/\sum_{i=1}^l \Z h_i$ is torsion free, and
\item $\Crt_{\Z}^\ast:=\mathrm{Hom}_{\Z}(\Crt_{\Z},\Z)$ contains $\Pi$.
\end{itemize}

We call $\Crt_{\Z}^\ast$ the weight lattice. Clearly, it is $W$-stable. We choose fundamental weights $\om_i \in \Crt_{\Z}^\ast$, $1\le i\le l$ satisfying $\om_i(h_j)=\de_{i,j}$ for all $1\le i,j\le l$. Note that if $\mathrm{rk} A=l$, then the $\om_i$'s are uniquely determined.

\subsection{Localization}
Following  \cite[(1.3)]{KK90} let $G$ be a Kac-Moody group, let $P_\Th$ be a parabolic subgroup containing the Borel subgroup $B$ and the compact maximal torus $T$, 
and let $G/P_\Th$ be the associated Kac-Moody flag variety. We now recall computations for both the $T$-equivariant Chow group and the equivariant $K$-theory of  $G/P_\Th$ with coefficients in a commutative ring $R$. 

Set $\La$ to be the character group of $T$. Note that $\Phi \subset \La \subset \Crt_{\Z}^\ast$. Assume that $\La$ is a formal Demazure lattice in the sense of \cite[Def.3.1]{Le}, i.e., every simple root of $\Pi$ can be extended to a $\Z$-basis of $\La$. We refer to \cite[\S3]{Le} for properties and examples of Demazure lattices.

Set $S=\Z[\La]$. Consider the left $S$-module $S_W=S\otimes_R R[W]$. Each element of $S_W$ can be written as an $S$-linear combination $\sum_{w\in W} q_w\de_w$, where $\{\de_w\}_{w\in W}$ is the standard basis and $q_w\in S$ are coefficients. The twisted commuting relation $w(q)\de_w =\de_w q$,  $q\in S$ induces a multiplication on $S_W$, hence, turning it into the twisted group algebra of \cite[(2.1)]{KK90}. 

Consider the localizations $Q=S[\tfrac{1}{x_\al},\al\in \Phi_+]$ and  $Q_W=Q \otimes_R R[W]$. Let $Y$ be the $R$-subalgebra of $Q_W$ generated by elements of $S\subset Q_W$ and by the push-pull elements 
\[
y_i=\tfrac{1}{x_{-{\al_i}}}+\tfrac{1}{x_{\al_i}}\de_{s_{i}}\in Q,
\] 
for all simple roots $\al_i\in \Pi$ and the corresponding  reflections $s_i$. Given $w\in W$ and its reduced word $w=s_{i_1}s_{i_2}\ldots s_{i_m}$ we set $y_w=y_{i_1}\ldots y_{i_m}$ in $Q_W$ (see \cite[(2.4)]{KK90}). According to \cite[(2.9)]{KK90} the elements $\{y_w\}_w$ form an $S$-basis of the algebra $Y$.

Let $\Psi$ be the $S$-linear dual of $Y$ and let $\om$ denote the $Q$-linear dual of $Q_W$.  Then $\Psi$ is an $S$-subalgebra of $\om$ (see \cite[Prop.2.20]{KK90}) which can be identified with the $T$-equivariant $K$-theory $K_T(G/B)$ by \cite[Thm.3.13]{KK90}. Moreover, the parabolic analogue of this result \cite[Cor.3.20]{KK90} says that the invariant subring $\Psi^{\Th}$ (under the Hecke action by $W_\Th$) can be identified with $K_T(G/P_\Th)$.

\subsection{The forgetful map}
We now explain the construction of the forgetful map in the context of localization. We assume for simplicity $\Th=\emptyset$.

Consider the $Q_W$ action on $Q$ defined by
\[
(\sum_w q_w\de_w) \cdot q'=\sum_w q_w w(q').
\] 
Using this action one identifies the algebra $Y$ with its image in $\mathrm{End}_R(S)$. Then composing with the augmentation $\ep\colon S\to R$ one obtains the map $Y \to \Hom_R(S,R)$
\[
d=\sum_w q_w\de_w =\{s\mapsto \sum_w q_ww(s)\}\quad\mapsto\quad \ep d=\{s\mapsto \ep(\sum_w q_ww(s))\}.
\]
The image of this map is denoted $\ep Y$. 

Consider $g\in \Psi$. It can be viewed as an element of $\Hom_S(Y,S)$ as follows:  if we write $g=(s_w)_w$, then $g\colon d=\sum_w q_w\de_w\mapsto \sum_w q_ws_w\in S$. 

We define the forgetful map $\rho\colon \Psi \to \ep \Psi=\Hom_R(\ep Y,R)$  by 
\[
g\mapsto \{\ep d \mapsto \ep g(d)\},
\] 
where $\{\ep d \mapsto \ep g(d)\} \in \ep \Psi$ sends
\[
\{s\mapsto \ep(\sum_w q_w w(s))\} \mapsto \ep(\sum_w q_ws_w).
\]

Let $y_{w}^\ast$ be the $Q$-linear dual of $y_{w}$. The set $\{y_{w}^\ast\}_w$ forms an $S$-basis of $\Psi$. By definition the forgetful map $\rho$ sends $y_{I_w}^*$ to  $\ep y_{v} \mapsto \ep (y_{w}^\ast(y_{v}))=\de_{w,v}$ which is the $R$-dual of $\ep y_{w}$. Hence, it maps 
\[
x=\sum_w s_w y_{w}^\ast \mapsto \rho(x)=\sum_w \ep(s_w) \ep y_{w}^\ast,
\] 
where $\{\ep y_{w}^\ast\}_w$ is the $R$-basis of $\ep\Psi$.

Therefore, we can identify $\ep\Psi$ with the quotient $\Psi/\mathcal{I}\Psi=\Z\otimes_S \Psi$ so that $\rho$ turns into the quotient map by \cite[Thm.3.28]{KK90}
\[
\rho\colon  K_T(G/B)\ra K(G/B).
\]
As before, this description can be extended to the parabolic situation in which case we obtain the forgetful map
\[
\rho\colon  K_T(G/P_\Theta)\ra K(G/P_\Theta).
\]

\subsection{The equivariant Riemann-Roch type formula}

Following to \cite{HHH} we may identify the invariant subring $\Psi^{\Th}$ and, hence, the $T$-equivariant $K$-theory $K_T(G/P_\Theta)$, with the structure algebra $\SA_m(\Gg)$ of the corresponding (parabolic) moment graph $\Gg=\Gg(W^\Th)$, where $W^\Th$ is the subset of minimal left coset representatives of $W/W_\Theta$ as in Example~\ref{ex:Bruhat}.
Observe that the similar identification between the structure algebra $\SA_a(\Gg)$ and the $T$-equivariant Chow ring $CH_T(G/P_\Theta)$ is also well-known.

Modulo all these identifications we obtain the following consequences of Proposition~\ref{thm:main} and Lemma~\ref{lem:forgt}:

\begin{cor}
The localized Chern character $\cch_i$ on the structure algebras of $\Gg(W^\Th)$ defines the the respective localized Chern character
\[
\cch_i\colon K_T(G/P_\Th) \to CH_T^{\le i}(G/P_\Th;\Q).
\]

Moreover, it restricts to the  the usual (non equivariant) Chern character 
\[
\tilde{\cch}_i\colon K(G/P_\Th) \to CH^{\le i}(G/P_\Th;\Q)
\]
so that  there is a commutative diagram
\[
\xymatrix{
K_T(G/P_\Th) \ar[d]_{\cch_i} \ar[r]^\rho & K(G/P_\Th)  \ar[d]^{\tilde{\cch}_i} \\
CH_T^{\le i}(G/P_\Th;\Q) \ar[r]^\rho  & CH^{\le i}(G/P_\Th;\Q)  
}.
\]
\end{cor}

We set $\xi$ to act by elements of $W$ as in Example~\ref{ex:xicharmap}.
Then by Example~\ref{ex:fibK} there is the $\xi$-fibration $\Gg(W) \to \Gg(W^\Theta)$ which is regular by \cite[Lemma~2.2]{CZZ}, since $\La$ is the Demazure lattice. 
We have $-\xi_y(\be)=-y(\beta)=vu(-\beta)$, where $y=vu$, $v\in W^\Theta$ and $u\in W_\Theta$. So
\[
td_i^\xi(\Gg)=\big(\exp\big(\sum_{\be\in N_y^\xi}-\xi_y(\be)\big)\big)_{\! y}=\big(\exp v\big(\sum_{\beta\in \Phi_+^\Theta\cap u(\Phi_-^\Theta)}\beta \big)\big)_{\! vu} \in CH^{\le i}_T(G/P_\Th;\Q)  
\]
and we obtain the following consequence of Theorem~\ref{thm:mainRR}:

\begin{cor}
For any $z\in K_T(G/P_\Theta)$ we have
\[
\pi^\xi_\ast \big(\cch_i(z)\cdot td_i^\xi(\Gg)\big)=\cch_i(\pi^\xi_\ast(z)).
\] 
\end{cor}


\bibliographystyle{alpha}

\end{document}